\documentclass[preprint,11pt]{elsarticle}

\usepackage{amssymb}
\usepackage{amsthm}
\usepackage{tikz}
\usepackage{tikz,pgfplots}
\usetikzlibrary{decorations.markings}

\usepackage[margin=2.5cm]{geometry}

\journal{Discrete Applied Mathematics}

\begin{document}

\newtheorem{definition}{Definition}
\newtheorem{theorem}{Theorem}
\newtheorem{corollary}{Corollary}
\newtheorem{lemma}{Lemma}
\newtheorem{conjecture}{Conjecture}

\tikzset{middlearrow/.style={
        decoration={markings,
            mark= at position 0.7 with {\arrow[scale=2]{#1}} ,
        },
        postaction={decorate}
    }
}

\begin{frontmatter}


\title{On diregular digraphs with degree two and excess two}



\author{James Tuite}

\address{Department of Mathematics and Statistics, Open University, Walton Hall, Milton Keynes}
\ead{james.tuite@open.ac.uk}

\begin{abstract}
An important topic in the design of efficient networks is the construction of $(d,k,+\epsilon )$-digraphs, i.e. $k$-geodetic digraphs with minimum out-degree $\geq d$ and order $M(d,k)+ \epsilon $, where $M(d,k)$ represents the Moore bound for degree $d$ and diameter $k$ and $\epsilon > 0$ is the (small) excess of the digraph.  Previous work has shown that there are no $(2,k,+1)$-digraphs for $k \geq 2$.  In a separate paper, the present author has shown that any $(2,k,+2)$-digraph must be diregular for $k \geq 2$.  In the present work, this analysis is completed by proving the nonexistence of diregular $(2,k,+2)$-digraphs for $k \geq 3$ and classifying diregular $(2,2,+2)$-digraphs up to isomorphism.

\end{abstract}


\end{frontmatter}

\section{Introduction}
\label{S:1}

An important topic in the design of interconnection networks is the directed degree/diameter problem: what is the largest possible order $N(d,k)$ of a digraph $G$ with maximum out-degree $d$ and diameter $\leq k$?  A simple inductive argument shows that for $0 \leq l \leq k$ the number of vertices at distance $l$ from a fixed vertex $v$ is bounded above by $d^l$.  Therefore, a natural upper bound for the order of such a digraph is the so-called \textit{Moore bound} $M(d,k) = 1 + d+ d^2 + ... + d^k$.  A digraph that attains this upper bound is called a \textit{Moore digraph}.  It is easily seen that a digraph $G$ is Moore if and only if it is out-regular with degree $d$, has diameter $k$ and is $k$-geodetic, i.e. for any two vertices $u, v$ there is at most one $\leq k$-path from $u$ to $v$.

As it was shown by Bridges and Toueg in \cite{BriTou} that Moore digraphs exist only in the trivial cases $d = 1$ or $k = 1$ (the Moore digraphs are directed $(k+1)$-cycles and complete digraphs $K_{d+1}$ respectively), much research has been devoted to the study of digraphs that in some sense approximate Moore digraphs.  For example, there is an extensive literature on digraphs with maximum out-degree $d$, diameter $\leq k$ and order $M(d,k) - \delta $ for small $\delta > 0$; this is equivalent to relaxing the $k$-geodecity requirement in the conditions for a digraph to be Moore.  $\delta $ is known as the \textit{defect} of the digraph.  The reader is referred to the survey \cite{MilSir2} for more information.

In this paper, however, we will consider the following related problem, which is obtained by retaining the $k$-geodecity requirement in the above characterisation of Moore digraphs, but allowing the diameter to exceed $k$: what is the smallest possible order of a $k$-geodetic digraph $G$ with minimum out-degree $\geq d$?  A $k$-geodetic digraph with minimum out-degree $\geq d$ and order $M(d,k) + \epsilon $ is said to be a $(d,k,+\epsilon )$-digraph or to have \textit{excess} $\epsilon $.  It was shown in \cite{Sil} that there are no diregular $(2,k,+1)$-digraphs for $k \geq 2$.  In 2016 it was shown in \cite{MirSil} that digraphs with excess one must be diregular and that there are no $(d,k,+1)$-digraphs for $k = 2,3,4$ and sufficiently large $d$.  In a separate paper \cite{Tui}, the present author has shown that $(2,k,+2)$-digraphs must be diregular with degree $d = 2$ for $k \geq 2$.  In the present paper, we classify the $(2,2,+2)$-digraphs up to isomorphism and show that there are no diregular $(2,k,+2)$-digraphs for $k \geq 3$, thereby completing the proof of the nonexistence of digraphs with degree $d = 2$ and excess $\epsilon = 2$ for $k \geq 3$.  Our reasoning and notation will follow closely that employed in \cite{MilSir} for the corresponding result for defect $\delta = 2$.

\section{Preliminary results}

We will let $G$ stand for a $(2,k,+2)$-digraph for arbitrary $k \geq 2$, i.e. $G$ has minimum out-degree $d = 2$, is $k$-geodetic and has order $M(2,k)+2$.  We will denote the vertex set of $G$ by $V(G)$.  By the result of \cite{Tui}, $G$ must be diregular with degree $d = 2$ for $k \geq 2$.  The distance $d(u,v)$ between vertices $u$ and $v$ is the length of the shortest path from $u$ to $v$.  Notice that $d(u,v)$ is not necessarily equal to $d(v,u)$.  $u \rightarrow v$ will indicate that there is an arc from $u$ to $v$.  We define the in- and out-neighbourhoods of a vertex $u$ by $N^-(u) = \{ v \in V(G) : v \rightarrow u\} $  and $N^+(u) = \{ v \in V(G) : u \rightarrow v\} $ respectively; more generally, for $0 \leq l \leq k$, the set $\{ v \in V(G) : d(u,v) = l\} $ of vertices at distance exactly $l$ from $u$ will be denoted by $N^l(u)$.  For $0 \leq l \leq k$ we will also write $T_l(u) = \cup ^{l}_{i=0}N^i(u)$ for the set of vertices at distance $\leq l$ from $u$.  The notation $T_{k-1}(u)$ will be abbreviated by $T(u)$.

It is easily seen that for any vertex $u$ of $G$, there are exactly two distinct vertices that are at distance $\geq k + 1$ from $u$.  For any $u \in V(G)$, we will write $O(u)$ for the set of these vertices and call such a set an \textit{outlier set} and its elements \textit{outliers} of $u$.  Notice that $O(u) = V(G) - T_k(u)$.  An elementary counting argument shows that in a diregular $(2,k,+2)$-digraph every vertex is also an outlier of exactly two vertices.  We will say that a vertex $u$ can \textit{reach} a vertex $v$ if $v \not \in O(u)$.  

Our proof will proceed by an analysis of a pair of vertices with exactly one common out-neighbour.  First, we must show that such a pair exists and deduce some elementary properties of pairs of vertices with identical out-neighbourhoods.

\begin{lemma}\label{identical neighbourhoods lemma}
For $k \geq 2$, let $u$ and $v$ be distinct vertices such that $N^+(u) = N^+(v) = \{ u_1,u_2\} $.  Then $u_1 \in O(u_2), u_2 \in O(u_1)$ and there exists a vertex $x$ such that $O(u) = \{ v,x\} , O(v) = \{ u,x\} $.  
\end{lemma}
\begin{proof}
Suppose that $u$ can reach $v$ by a $\leq k$-path.  Then $v \in T(u_1) \cup T(u_2)$.  As $N^+(v) = N^+(u)$, it follows that there would be a $\leq k$-cycle through $v$, contradicting $k$-geodecity.  If $O(u) = \{ v,x\} $, then $x \not = v$ and $x \not \in T(u_1) \cup T(u_2)$, so that $v$ cannot reach $x$ by a $\leq k$-path.  Similarly, if $u_1$ can reach $u_2$ by a $\leq k$-path, then we must have $\{ u,v\} \cap T(u_1) \not = \varnothing $, which is impossible. 
\end{proof}

\begin{lemma}\label{vertices with one common neighbour} 
For $k \geq 2$, there exists a pair of vertices $u,v$ with $|N^+(u) \cap N^+(v)| = 1$.
\end{lemma}
\begin{proof}
Suppose for a contradiction that there is no such pair of vertices.  Define a map $\phi : V(G) \rightarrow V(G)$ as follows.  Let $u^+$ be an out-neighbour of a vertex $u$ and let $\phi (u)$ be the in-neighbour of $u^+$ distinct from $u$.  By our assumption, it is easily verified that $\phi $ is a well-defined bijection with no fixed points and with square equal to the identity.  It follows that $G$ must have even order, whereas $|V(G)| = M(2,k)+2$ is odd.  
\end{proof}
$u, v$ will now stand for a pair of vertices with a single common out-neighbour.  We will label the vertices of $T_k(u)$ according to the scheme $N^+(u) = \{ u_1,u_2\} , N^+(u_1) = \{ u_3, u_4\} , N^+(u_2) = \{ u_5, u_6\} , N^+(u_3) = \{ u_7,u_8\} , N^+(u_4) = \{ u_9,u_{10}\} $ and so on, with the same convention for the vertices of $T_k(v)$, where we will assume that $u_2 = v_2$.

\section{Classification of $(2,2,+2)$-digraphs}
We begin by classifying the $(2,2,+2)$-digraphs up to isomorphism.  We will prove the following theorem.

\begin{theorem}
There are exactly two diregular $(2,2,+2)$-digraphs, which are displayed in Figures \ref{fig:badpairdigraph} and \ref{fig:firstgooddigraph}.
\end{theorem}  

Let $G$ be an arbitrary diregular $(2,2,+2)$-digraph.  $G$ has order $M(2,2)+2 = 9$.  By Lemma \ref{vertices with one common neighbour}, $G$ contains a pair of vertices $(u,v)$ such that $|N^+(u) \cap N^+(v)| = 1$; we will assume that $u_2 = v_2$, so that we have the situation shown in Figure \ref{fig:setup}.

\begin{figure}\centering
\begin{tikzpicture}[middlearrow=stealth,x=0.2mm,y=-0.2mm,inner sep=0.1mm,scale=1.85,
	thick,vertex/.style={circle,draw,minimum size=10,font=\tiny,fill=white},edge label/.style={fill=white}]
	\tiny
	\node at (150,0) [vertex] (v0) {$u$};
	\node at (100,50) [vertex] (v1) {$u_1$};
	\node at (200,50) [vertex] (v2) {$u_2$};
	\node at (300,50) [vertex] (v3) {$v_1$};
	\node at (130,100) [vertex] (v4) {$u_4$};
           \node at (170,100) [vertex] (v5) {$u_5$};
	\node at (230,100) [vertex] (v6) {$u_6$};
	\node at (250,0) [vertex] (v7) {$v$};
          \node at (70,100) [vertex] (v9) {$u_3$};
           \node at (270,100) [vertex] (v10) {$v_3$};
          \node at (330,100) [vertex] (v11) {$v_4$};

	\path
		(v0) edge [middlearrow] (v1)
                     (v0) edge [middlearrow] (v2)
                     (v1) edge [middlearrow] (v9)
                     (v1) edge [middlearrow] (v4)
                     (v2) edge [middlearrow] (v5)
                     (v2) edge [middlearrow] (v6)
                     (v7) edge [middlearrow] (v2)
                     (v7) edge [middlearrow] (v3)
                     (v3) edge [middlearrow] (v10)
                     (v3) edge [middlearrow] (v11)
				
		;
\end{tikzpicture}
\caption{The vertices $u$ and $v$}
\label{fig:setup}
\end{figure}
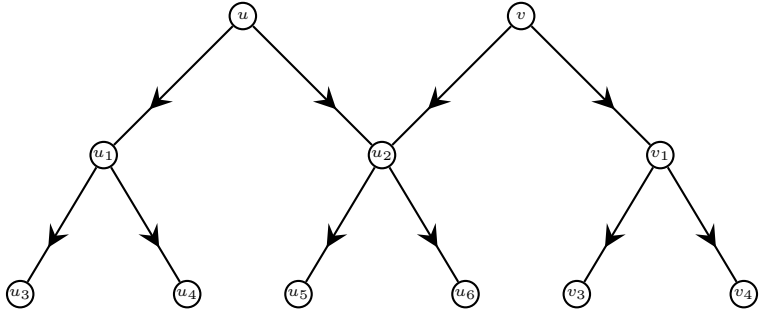

We can immediately deduce some information on the possible positions of $v$ and $v_1$ in $T_2(u)$.

\begin{lemma}\label{k equals 2 position of v1}
If $v \not \in O(u)$, then $v \in N^+(u_1)$.  If $v_1 \not \in O(u)$, then $v_1 \in N^+(u_1)$.
\end{lemma}
\begin{proof}
$v \not \in T(u_2)$ by $2$-geodecity.  $v \not = u$ by construction.  If we had $v = u_1$, then there would be two distinct $\leq 2$-paths from $u$ to $u_2$. Also $v_1 \not \in \{ u\} \cup T(u_2)$ by 2-geodecity and by assumption $u_1 \not = v_1$.
\end{proof} 

Since $v$ and $v_1$ cannot both lie in $N^+(u_1)$ by 2-geodecity, we have the following corollary.

\begin{corollary}\label{k equals 2 v or v1 an outlier}
$O(u) \cap \{ v,v_1\} \not = \varnothing $.
\end{corollary}

We will call a pair of vertices $(u,v)$ with a single common out-neighbour \textit{bad} if at least one of 
\[ O(u) \cap \{ v_1,v_3\} = \varnothing , O(u) \cap \{ v_1,v_4\} = \varnothing ,  O(v) \cap \{ u_1,u_3\} = \varnothing , O(v) \cap \{u_1,u_4\} = \varnothing .\]
holds.  Otherwise such a pair will be called \textit{good}.

\begin{lemma}
There is a unique $(2,2,+2)$-digraph containing a bad pair.
\end{lemma}
\begin{proof}
Assume that there exists a bad pair $(u,v)$.  Without loss of generality, $O(u) \cap \{ v_1,v_3\} = \varnothing $.  By Lemma \ref{k equals 2 position of v1} we can set $v_1 = u_3$.  By 2-geodecity $v_3 = u$.  We cannot have $v_4 = v_3 = u$, so $v_4$ must be an outlier of $u$.  By Corollary \ref{k equals 2 v or v1 an outlier} it follows that $O(u) = \{ v,v_4\} $.

Consider the vertex $u_1$.  By Lemma \ref{k equals 2 position of v1}, if $u_1 \not \in O(v)$, then $u_1 \in N^+(v_1)$.  However, as $v_1 = u_3$, there would be a 2-cycle through $u_1$.  Hence $u_1 \in O(v)$.  As $O(u) = \{ v,v_4\} $, we have $V(G) = \{ u,u_1,u_2,u_3=v_1,u_4,u_5,u_6,v,v_4\} $ and $O(v) = \{ u_1,u_4\} $.  As neither $u$ nor $v$ lies in $T(u_1)$, we must also have $u_2 \in O(u_1)$.  As $u_1$ can reach $u_1, v_1,u_4,u$ and $v_4$, it follows that without loss of generality we either have $O(u_1) = \{ u_2,v\} $ and $N^+(u_4) = \{ u_5,u_6\} = N^+(u_2)$ or $O(u_1) = \{ u_2,u_6\}$ and $N^+(u_4) = \{ v,u_5\} $.  In either case, $(v,u_1)$ is a good pair.

Suppose firstly that $N^+(u_2)= N^+(u_4)$.  Then $v$ is an outlier of $u$ and $u_1$.  As each vertex is the outlier of exactly two vertices, $v_1$ must be able to reach $v$ by a $\leq 2$-path.  Hence $v_4 \rightarrow v$.  Likewise $u_2$ can reach $v$, so without loss of generality $u_5 \rightarrow v$.  Suppose that $O(u_2) \cap \{ u,u_1\} = \varnothing $.  As $u$ and $v$ have a common out-neighbour, we must have $u_6 \rightarrow u$.  Since $u \rightarrow u_1$, by 2-geodecity we must have $u_5 \rightarrow u_1$.  However, this is a contradiction, as $v$ and $u_1$ also have a common out-neighbour.  Therefore, at least one of $u,u_1$ is an outlier of $u_2$. By Lemma \ref{identical neighbourhoods lemma} $u_4$ is an outlier of $u_2$.  Therefore either $O(u_2) = \{ u,u_4\} $ or $O(u_2) = \{ u_1,u_4\} $.  If $O(u_2) = \{ u,u_4\} $, then $u_2$ must be able to reach $u_1,v_1$ and $v_4$.  $u_5 \rightarrow v$ and $v \rightarrow v_1$, so $v_1 \in N^+(u_6)$.  As $u_1 \rightarrow v_1$, we must have $N^+(u_5) = \{ v,u_1\} $.  As $v$ and $u_1$ have a common out-neighbour, this violates 2-geodecity.  Hence $O(u_2) = \{ u_1,u_4\} $ and $u_2$ can reach $u, v_1$ and $v_4$.  As $v \rightarrow v_1$, $v_1 \in N^+(u_6)$.  As $v_1 \rightarrow v_4$, it follows that $N^+(u_5) = \{ v,v_4\} $.  However, $v_4 \rightarrow v$, so this again violates 2-geodecity. 

We are left with the case $O(u_1) = \{ u_2,u_6\}$ and $N^+(u_4) = \{ v,u_5\} $.  Then $v_1 \in O(u_2)$, as neither $v$ nor $u_1$ lies in $T(u_2)$.  Observe that $u_2$ and $u_4$ have a single common out-neighbour, so by Corollary \ref{k equals 2 v or v1 an outlier} $O(u_2) \cap \{ u_4,v\} \not = \varnothing $.  Therefore either $O(u_2) = \{ v_1,u_4\} $ or $O(u_2) = \{ v_1,v\} $.  Suppose firstly that $O(u_2) = \{ v_1,u_4\} $.  Then $N^2(u_2) = \{ u,v,u_1,v_4 \} $.  As $N^+(u_4) = \{ v,u_5\} $, $u_5 \not \rightarrow v$, so $u_6 \rightarrow v$.  As $N^+(u) \cap N^+(v) \not = \varnothing $, $u_5 \rightarrow u$.  $u \rightarrow u_1$, so necessarily $N^+(u_6) = \{ v,u_1\} $.  However, $v_1 \in N^+(u_1) \cap N^+(v)$, contradicting 2-geodecity.  

Hence $O(u_2) = \{ v_1,v\} $ and $N^2(u_2) = \{ u,u_1,u_4,v_4\} $.  As $u_4 \rightarrow u_5$, $u_5 \not \rightarrow u_4$.  Thus $u_6 \rightarrow u_4$.  Now $u_1 \rightarrow u_4$ and $u \rightarrow u_1$ implies that $N^+(u_5) = \{ u_1,v_4\} $ and $N^+(u_6) = \{ u,u_4\} $.  Finally we must have $N^+(v_4) = \{ v,u_6\} $.  This gives us the $(2,2,+2)$-digraph shown in Figure \ref{fig:badpairdigraph}.     

\end{proof}

\begin{figure}\centering
\begin{tikzpicture}[middlearrow=stealth,x=0.2mm,y=-0.2mm,inner sep=0.1mm,scale=1.85,
	thick,vertex/.style={circle,draw,minimum size=10,font=\tiny,fill=white},edge label/.style={fill=white}]
	\tiny
	\node at (200,0) [vertex] (v0) {$u$};
	\node at (290,130) [vertex] (v1) {$u_1$};
	\node at (250,30) [vertex] (v2) {$u_2$};
	\node at (110,130) [vertex] (v3) {$v_1$};
	\node at (250,190) [vertex] (v4) {$u_4$};
           \node at (330,170) [vertex] (v5) {$u_5$};
	\node at (150,30) [vertex] (v6) {$u_6$};
	\node at (150,190) [vertex] (v7) {$v$};
	\node at (70,170) [vertex] (v8) {$v_4$};

	\path
		(v0) edge [middlearrow] (v1)
		(v0) edge [middlearrow] (v2)
		(v1) edge [middlearrow] (v3)
                     (v1) edge [middlearrow] (v4)
                     (v2) edge [middlearrow] (v5)
                     (v2) edge [middlearrow] (v6)
                     (v3) edge [middlearrow] (v0)
                     (v3) edge [middlearrow] (v8)
	          (v4) edge [middlearrow] (v5)
                     (v4) edge [middlearrow] (v7)
                     (v5) edge [middlearrow] (v1)
                     (v5) edge [middlearrow] (v8)
                     (v6) edge [middlearrow] (v0)
                     (v6) edge [middlearrow] (v4)
                     (v7) edge [middlearrow] (v2)
                     (v7) edge [middlearrow] (v3)
                     (v8) edge [middlearrow] (v6)
                     (v8) edge [middlearrow] (v7)
		
		;
\end{tikzpicture}
\caption{The unique $(2,2,+2)$-digraph containing a bad pair}
\label{fig:badpairdigraph}
\end{figure}
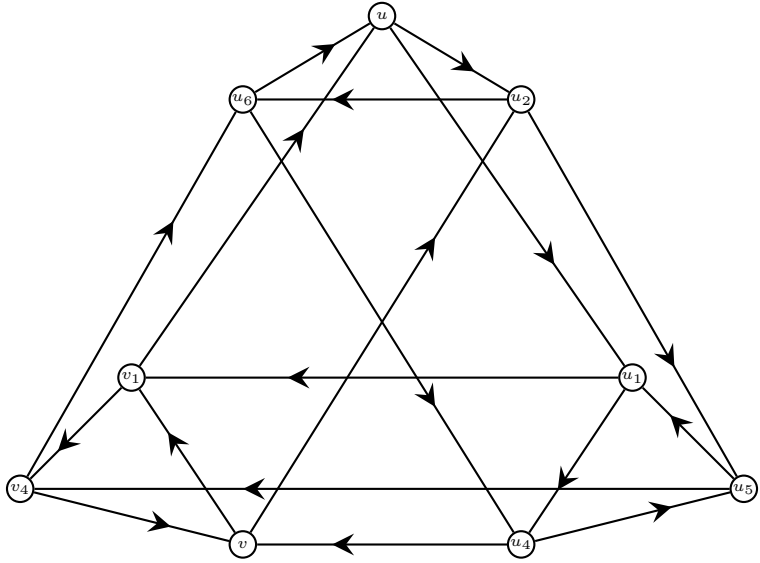

We can now assume that all pairs given by Lemma \ref{vertices with one common neighbour} are good.  Let us fix a pair $(u,v)$ with a single common out-neighbour.  It follows from Corollary \ref{k equals 2 v or v1 an outlier} and the definition of a good pair that $v_1 \in O(u)$; otherwise $O(u)$ would contain $v$, $v_3$ and $v_4$, which is impossible.  Likewise $u_1 \in O(v)$.  

Considering the positions of $v_3$ and $v_4$, we see that there are without loss of generality four possibilities: 1) $u = v_3, u_4 = v_4$, 2) $u = v_3, O(u) = \{ v_1,v_4\} $, 3) $N^+(u_1) = N^+(v_1)$ and 4) $u_3 = v_3, O(u) = \{ v_1,v_4\} $.  A suitable relabelling of vertices shows that case 4 is equivalent to case 1.a) below, so we will examine cases 1 to 3 in turn.
\newline
\newline
{\bf Case 1:} $\mathbf{u = v_3, u_4 = v_4}$

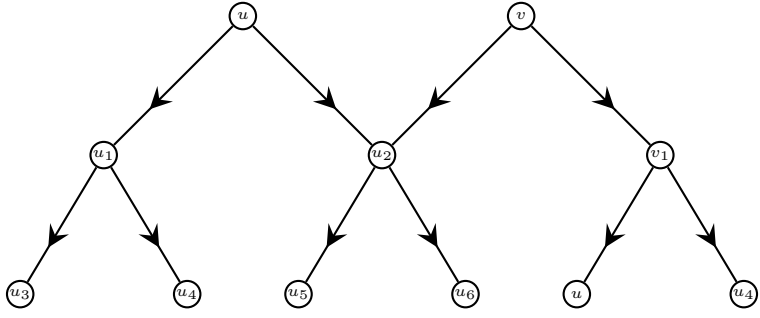
\begin{figure}\centering
\begin{tikzpicture}[middlearrow=stealth,x=0.2mm,y=-0.2mm,inner sep=0.1mm,scale=1.85,
	thick,vertex/.style={circle,draw,minimum size=10,font=\tiny,fill=white},edge label/.style={fill=white}]
	\tiny
	\node at (150,0) [vertex] (v0) {$u$};
	\node at (100,50) [vertex] (v1) {$u_1$};
	\node at (200,50) [vertex] (v2) {$u_2$};
	\node at (300,50) [vertex] (v3) {$v_1$};
	\node at (130,100) [vertex] (v4) {$u_4$};
           \node at (170,100) [vertex] (v5) {$u_5$};
	\node at (230,100) [vertex] (v6) {$u_6$};
	\node at (250,0) [vertex] (v7) {$v$};
          \node at (70,100) [vertex] (v9) {$u_3$};
           \node at (270,100) [vertex] (v10) {$u$};
          \node at (330,100) [vertex] (v11) {$u_4$};

	\path
		(v0) edge [middlearrow] (v1)
                     (v0) edge [middlearrow] (v2)
                     (v1) edge [middlearrow] (v9)
                     (v1) edge [middlearrow] (v4)
                     (v2) edge [middlearrow] (v5)
                     (v2) edge [middlearrow] (v6)
                     (v7) edge [middlearrow] (v2)
                     (v7) edge [middlearrow] (v3)
                     (v3) edge [middlearrow] (v10)
                     (v3) edge [middlearrow] (v11)
				
		;
\end{tikzpicture}
\caption{Case 1 configuration}
\label{fig:case1}
\end{figure}
Depending upon the position of $v$, we must either have $O(u) = \{ v_1,v\} $ and $O(v) = \{ u_1,u_3\} $ or $v = u_3$.
\newline
\newline
{\bf Case 1.a):} $\mathbf{ O(u) = \{ v_1,v\} , O(v) = \{ u_1,u_3\} }$

In this case $V(G) = \{ u,u_1,u_2,u_3,u_4,u_5,u_6,v,v_1\} $.  $u_1$ and $v_1$ have a single common out-neighbour, namely $u_4$, so, as we are assuming all such pairs to be good, we have $u_3 \in O(v_1), u \in O(u_1)$.  By 2-geodecity, $N^+(u_4) \subset \{ u_5,u_6,v\} $, so without loss of generality either $N^+(u_4) = \{ u_5,u_6\} $ or $N^+(u_4) = \{ u_5,v\} $.

Suppose that $N^+(u_4) = \{ u_5,u_6\} $.  By elimination, $O(v_1) = \{ v,u_3\} $.  As $G$ is diregular, every vertex is an outlier of exactly two vertices; $v$ is an outlier of $u$ and $v_1$, so both $u_1$ and $u_2$ can reach $v$ by a $\leq 2$-path.  Hence $v \in N^+(u_3)$.  As $v \rightarrow v_1$, we see that $v_1$ is an outlier of $u_1$; as $u$ is also an outlier of $u_1$, we have $O(u_1) = \{ u,v_1\} $ and $N^+(u_3) = \{ v,u_2\} $.  As $v \rightarrow u_2$, this is impossible.

Now consider $N^+(u_4) = \{ u_5,v\} $.  We now have $O(v_1) = \{ u_3,u_6\} $.  Thus $u_3 \in O(v) \cap O(v_1)$, so $u_3 \in T_2(u_4)$.  $v$ is not adjacent to $u_3$, so $u_3 \in N^+(u_5)$.  $u_2$ and $u_4$ have $u_5$ as a unique common out-neighbour, so $u_6 \in O(u_4), v \in O(u_2)$.  As $u_6 \in O(v_1) \cap O(u_4)$, $u_1$ can reach $u_6$.  Hence $u_6 \in N^+(u_3)$.  Neither $u$ nor $v$ lie in $T(u_1)$, so $u_2 \in O(u_1)$.  Therefore either $O(u_1) = \{ u,u_2\} $ or $O(u_1) = \{ u_2,v_1\} $.  If $O(u_1) = \{ u,u_2\} $, then $N^+(u_3) = \{ u_6,v_1\} $.  $u_2$ can't reach $v_1$, since $v,u_3 \not \in T(u_2)$, so $O(u_2) = \{ v,v_1\} $ and $N^2(u_2) = \{ u,u_1,u_3,u_4\} $.  As $u_4 \rightarrow u_5$, $u_4 \in N^+(u_6)$.  $u_1 \rightarrow u_4$, so $N^+(u_5) = \{ u_1,u_3\} $.  As $u_1 \rightarrow u_3$, this is a contradiction.  Thus $O(u_1) = \{ u_2,v_1\} $, so that $N^+(u_3) = \{ u,u_6\} $.  $u_1$ must have an in-neighbour apart from $u$, which must be either $u_5$ or $u_6$.  As $u_1 \rightarrow u_3$, we have $u_1 \in N^+(u_6)$.  By elimination, $v$ and $v_1$ must also have in-neighbours in $\{ u_5,u_6\} $.  As $u_1$ and $v_1$ have a common out-neighbour, we have $N^+(u_5) = \{ u_3,v_1\} , N^+(u_6) = \{ u_1,v\} $.  However, both $u_3$ and $v_1$ are adjacent to $u$, violating 2-geodecity.
\newline
\newline
{\bf Case 1.b):} $\mathbf{ v = u_3}$

There exists a vertex $x$ such that $V(G) = \{ u,u_1,u_2,v,u_4,u_5,u_6,v_1,x\} $, $O(u) = \{ v_1,x\} $ and $O(v) = \{ u_1,x\} $.  As $x \in O(u) \cap O(v)$, $u_1$ and $u_2$ can reach $x$, so without loss of generality $x \in N^+(u_4) \cap N^+(u_5)$.  As $u_5$ and $u_4$ have a common out-neighbour, $u_5 \in O(u_1)$.  Also, $u_1$ and $v_1$ have $u_4$ as a unique common out-neighbour, so $u \in O(u_1)$ and $O(u_1) = \{ u,u_5\} $.  Thus $N^+(u_4) = \{ x,u_6\} $.  Observe that $u_2$ and $u_4$ have the out-neighbour $u_6$ in common.  Thus $x \in O(u_2)$, whereas we already have $x \in O(u) \cap O(v)$, a contradiction.
\newline
\newline
{\bf Case 2:} $\mathbf{ u = v_3, O(u) = \{ v_1,v_4\} }$

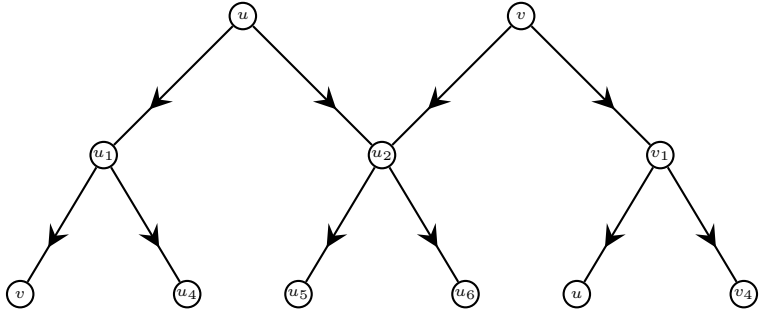
\begin{figure}\centering
\begin{tikzpicture}[middlearrow=stealth,x=0.2mm,y=-0.2mm,inner sep=0.1mm,scale=1.85,
	thick,vertex/.style={circle,draw,minimum size=10,font=\tiny,fill=white},edge label/.style={fill=white}]
	\tiny
	\node at (150,0) [vertex] (v0) {$u$};
	\node at (100,50) [vertex] (v1) {$u_1$};
	\node at (200,50) [vertex] (v2) {$u_2$};
	\node at (300,50) [vertex] (v3) {$v_1$};
	\node at (130,100) [vertex] (v4) {$u_4$};
           \node at (170,100) [vertex] (v5) {$u_5$};
	\node at (230,100) [vertex] (v6) {$u_6$};
	\node at (250,0) [vertex] (v7) {$v$};
          \node at (70,100) [vertex] (v9) {$v$};
           \node at (270,100) [vertex] (v10) {$u$};
          \node at (330,100) [vertex] (v11) {$v_4$};

	\path
		(v0) edge [middlearrow] (v1)
                     (v0) edge [middlearrow] (v2)
                     (v1) edge [middlearrow] (v9)
                     (v1) edge [middlearrow] (v4)
                     (v2) edge [middlearrow] (v5)
                     (v2) edge [middlearrow] (v6)
                     (v7) edge [middlearrow] (v2)
                     (v7) edge [middlearrow] (v3)
                     (v3) edge [middlearrow] (v10)
                     (v3) edge [middlearrow] (v11)
				
		;
\end{tikzpicture}
\caption{Case 2 configuration}
\label{fig:case2}
\end{figure}

As $v$ is not equal to $v_1$ or $v_4$, $v$ must lie in $T_2(u)$.  Without loss of generality, $v = u_3$.  Hence $V(G) = \{ u,u_1,u_2,v,u_4,u_5,u_6,v_1,v_4\} $ and $O(v) = \{ u_1,u_4\} $.  We have the configuration shown in Figure \ref{fig:case2}.  Hence $u_1$ can reach $u_1, v, u_4, u_2$ and $v_1$, so we have without loss of generality one of the following: a) $O(u_1) = \{ u,v_4\} , N^+(u_4) = \{ u_5,u_6\} $, b) $O(u_1) = \{ u,u_5\} , N^+(u_4) = \{ u_6,v_4\} $, c) $O(u_1) = \{ u_5,u_6\} , N^+(u_4) = \{ u,v_4\} $ or d) $O(u_1) = \{ u_5,v_4\} , N^+(u_4) = \{ u,u_6\} $.
\newline
\newline
{\bf Case 2.a):} $\mathbf{ O(u_1) = \{ u,v_4\} , N^+(u_4) = \{ u_5,u_6\} }$

As $v_4 \in O(u) \cap O(u_1)$, $u_2$ can reach $v_4$ and without loss of generality $v_4 \in N^+(u_5)$.  $N^+(u_2) = N^+(u_4)$, so by Lemma \ref{identical neighbourhoods lemma} $u_2 \in O(u_4), u_4 \in O(u_2), u_5 \in O(u_6)$ and $u_6 \in O(u_5)$.  Hence $u_4 \in O(v) \cap O(u_2)$, so $v_1$ can reach $u_4$, so $u_4 \in N^+(v_4)$.  Neither $u_5$ nor $u_6$ lies in $N^+(v_4)$, so $O(v_1) = \{ u_5,u_6\} $ and $N^+(v_4) = \{ u_4,v\} $.  Hence $O(v_4) = \{ u,u_1\} $.  Observe that $N^+(u_1) = N^+(v_4)$, so that $v \in O(u_4)$.  Therefore $v \not \in N^+(u_5) \cup N^+(u_6)$, yielding $O(u_2) = \{ u_4,v\} , N^2(u_2) = \{ v_4,v_1,u,u_1\} $.  As $v_1 \rightarrow v_4$ and $N^+(u_1) = N^+(v_4)$, we must have $N^+(u_5) = \{ v_4,u\} , N^+(u_6) = \{v_1,u_1\} $.  This yields the digraph in Figure \ref{fig:firstgooddigraph}.  Unlike the digraph in Figure \ref{fig:badpairdigraph}, this digraph contains pairs of vertices with identical out-neighbourhoods, so the two are not isomorphic. 
\newline
\newline
{\bf Case 2.b):} $\mathbf{ O(u_1) = \{ u,u_5\} , N^+(u_4) = \{ u_6,v_4\} }$

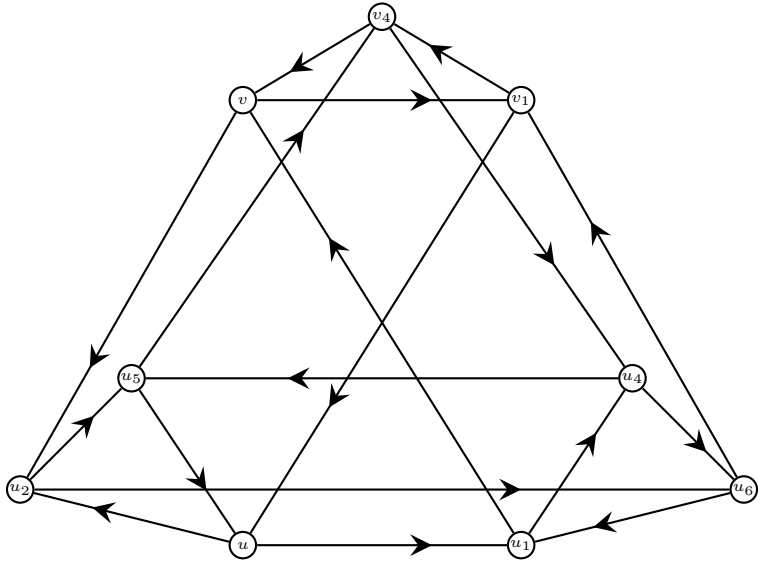
\begin{figure}\centering
\begin{tikzpicture}[middlearrow=stealth,x=0.2mm,y=-0.2mm,inner sep=0.1mm,scale=1.85,
	thick,vertex/.style={circle,draw,minimum size=10,font=\tiny,fill=white},edge label/.style={fill=white}]
	\tiny
	\node at (200,0) [vertex] (v8) {$v_4$};
	\node at (290,130) [vertex] (v4) {$u_4$};
	\node at (250,30) [vertex] (v7) {$v_1$};
	\node at (110,130) [vertex] (v5) {$u_5$};
	\node at (250,190) [vertex] (v1) {$u_1$};
           \node at (330,170) [vertex] (v6) {$u_6$};
	\node at (150,30) [vertex] (v3) {$v$};
	\node at (150,190) [vertex] (v0) {$u$};
	\node at (70,170) [vertex] (v2) {$u_2$};

	\path
		(v0) edge [middlearrow] (v1)
                     (v0) edge [middlearrow] (v2)
		(v1) edge [middlearrow] (v3)
		(v1) edge [middlearrow] (v4)
                     (v2) edge [middlearrow] (v5)
		(v2) edge [middlearrow] (v6)
                     (v3) edge [middlearrow] (v2)
		(v3) edge [middlearrow] (v7)
                     (v4) edge [middlearrow] (v5)
		(v4) edge [middlearrow] (v6)
                     (v5) edge [middlearrow] (v0)
		(v5) edge [middlearrow] (v8)
                     (v6) edge [middlearrow] (v1)
		(v6) edge [middlearrow] (v7)
                     (v7) edge [middlearrow] (v0)
		(v7) edge [middlearrow] (v8)
                     (v8) edge [middlearrow] (v3)
                     (v8) edge [middlearrow] (v4)

		;
\end{tikzpicture}
\caption{A second $(2,2,+2)$-digraph}
\label{fig:firstgooddigraph}
\end{figure}

As $u_4 \rightarrow v_4$, $u_4 \not \in N^+(v_4)$, so $u_4 \in O(v_1)$.  Hence $u_4 \in O(v) \cap O(v_1)$, so $u_2$ can reach $u_4$.  As $u_4 \rightarrow u_6$, we must have $u_5 \rightarrow u_4$.  $u_2$ and $u_4$ have $u_6$ as a common out-neighbour, so $v_4 \in O(u_2), u_5 \in O(u_4)$.  Therefore $v_4 \in O(u) \cap O(u_2)$, so that $u_6$ can reach $v_4$, but $v_4 \not \in T(u_6)$, so $N^+(u_6)$ contains an in-neighbour of $v_4$.  $u_4 \not \in N^+(u_6)$, so we must have $u_6 \rightarrow v_1$.  We have $u_5 \in O(u_4) \cap O(u_1)$, so $v_1$ can reach $u_5$ and hence $v_4 \rightarrow u_5$.  $v_1$ cannot reach $u_6$, as $u_2,u_4 \not \in T(v_1)$, so $O(v_1) = \{ u_4,u_6\} , N^+(v_4) = \{ u_5,v\} $.  Now $u_2$ and $v_4$ have $u_5$ as a unique common out-neighbour, so $u_6 \in O(v_4), v \in O(u_2)$.  Thus $O(u_2) = \{ v,v_4\} $ and $N^2(u_2) = \{ u_4,v_1,u,u_1\} $.  Taking into account adjacencies between members of $N^2(u_2)$, it follows that $N^+(u_5) = \{ u_4,u\} , N^+(u_6) = \{ u_1,v_1\}  $.  However, $(u_2,u_4)$ now constitutes a bad pair, contradicting our assumption.  
\newline
\newline
{\bf Case 2.c):} $\mathbf{ O(u_1) = \{ u_5,u_6\} , N^+(u_4) = \{ u,v_4\} }$

As $u_4 \rightarrow v_4$, $u_4 \not \in N^+(v_4)$.  Hence $u_4 \in O(v) \cap O(v_1)$, implying that $u_2$ can reach $u_4$.  Without loss of generality, $u_5 \rightarrow u_4$.  There are three possibilities: i) $O(v_1) = \{ u_4,u_6\} , N^+(v_4) = \{ v,u_5\} $, ii) $O(v_1) = \{ u_4,u_5\} , N^+(v_4) = \{ v,u_6\} $ and iii) $O(v_1) = \{ u_4,v\} , N^+(v_4) = \{ u_5,u_6\} $.

 i) $O(v_1) = \{ u_4,u_6\} , N^+(v_4) = \{ v,u_5\} $

$u_1$ and $v_4$ have $v$ as a unique common out-neighbour, so $u_4 \in O(v_4)$.  However, this contradicts $v_4 \rightarrow u_5 \rightarrow u_4$.

 ii) $O(v_1) = \{ u_4,u_5\} , N^+(v_4) = \{ v,u_6\} $

Neither $u_4$ nor $v_1$ lie in $T(u_2)$, so $v_4 \in O(u_2)$.  Now observe that $u_2$ and $v_4$ have $u_6$ as unique common out-neighbour, so $v \in O(u_2)$, yielding $O(u_2) = \{ v,v_4\} $ and $N^2(u_2) = \{ u_4,u_1,u,v_1\} $.  As $u_4 \rightarrow u$ and $u \rightarrow u_1$, we must have $N^+(u_5) = \{ u_4,u_1\} , N^+(u_6) = \{ u,v_1\} $, a contradiction, since $u_1 \rightarrow u_4$.

 iii) $O(v_1) = \{ u_4,v\} , N^+(v_4) = \{ u_5,u_6\} $

We now have $N^+(u_2) = N^+(v_4)$, so $u_2 \in O(v_4), v_4 \in O(u_2), u_5 \in O(u_6), u_6 \in O(u_5)$.  Also $N^+(u_4) = N^+(v_1)$, so $u_4 \in O(v_1), v_1 \in O(u_4)$ and $u \in O(v_4)$.  $u \in O(v_4)$ implies that $u \not \in N^+(u_5) \cup N^+(u_6)$, so we see that $u \in O(u_2)$ and hence $O(u_2) = \{ u,v_4\} $ and $N^2(u_2) = \{ u_4,u_1,v,v_1\}$.  As $u_1 \rightarrow u_4$ and $u_1 \rightarrow v$, we have $N^+(u_5) = \{ u_4,v\} , N^+(u_6) = \{ u_1,v_1\} $.  It is not difficult to show that this yields a $(2,2,+2)$-digraph isomorphic to that in Figure \ref{fig:firstgooddigraph}.
\newline
\newline
{\bf Case 2.d):} $\mathbf{ O(u_1) = \{ u_5,v_4\} , N^+(u_4) = \{ u,u_6\} }$

In this case $v_4 \in O(u) \cap O(u_1)$, so $u_2$ can reach $v_4$.  $u_4$ and $v_1$ have unique common out-neighbour $u$, so $v_4 \in O(u_4), u_6 \in O(v_1)$.  If $u_6 \rightarrow v_4$, then we would have $u_4 \rightarrow u_6 \rightarrow v_4$, contradicting $v_4 \in O(u_4)$, so $u_5 \rightarrow v_4$.  This also implies that $u_5 \not \in N^+(v_4)$, so $u_5 \in O(v_1)$, yielding $O(v_1) = \{ u_5,u_6\} $ and $N^+(v_4) = \{ v,u_4\} = N^+(u_1)$.  Now $v_4,u_1 \not \in T(u_2)$, so $O(u_2) = \{ v,u_4\} $ and $N^2(u_2) = \{ v_1,v_4,u,u_1\} $.  As $v_1 \rightarrow v_4$ and $v_1 \rightarrow u$, it follows that $N^+(u_5) = \{ v_4,u\} , N^+(u_6) = \{ u_1,v_1\} $.  However, we now have paths $u_4 \rightarrow u \rightarrow u_1$ and $u_4 \rightarrow u_6 \rightarrow u_1$, which is impossible.
\newline
\newline
{\bf Case 3:} $\mathbf{N^+(u_1) = N^+(v_1)}$

It is easy to see by 2-geodecity that $V(G) = \{ u,u_1,u_2,u_3,u_4,u_5,u_6,v,v_1\} $, $O(u) = \{ v,v_1\} $ and $O(v) = \{ u,u_1\} $.  As $u_1,v_1 \not \in T(u_2)$, we have $O(u_2) = \{ u_3,u_4\} $ and $N^2(u_2) = \{ u,u_1,v,v_1\} $.  Without loss of generality, $N^+(u_5) = \{ u,v_1\} , N^+(u_6) = \{ v,u_1\} $.  $u$ and $v$ have in-neighbours apart from $u_5$ and $u_6$ respectively, so without loss of generality $u_3 \rightarrow u, u_4 \rightarrow v$.  Likewise, $u_5$ and $u_6$ have in-neighbours other than $u_2$, so, as $u_5 \rightarrow u$ and $u_6 \rightarrow v$, we must have $N^+(u_3) = \{ u,u_6\} , N^+(u_4) = \{ v,u_5\} $.  But now we have paths $u_3 \rightarrow u \rightarrow u_1$ and $u_3 \rightarrow u_6 \rightarrow u_1$, violating 2-geodecity.    

\begin{corollary}
There is a unique $(2,2,+2)$-digraph containing no bad pairs.
\end{corollary}
This completes our analysis of diregular $(2,2,+2)$-digraphs.  As it was shown in \cite{Tui} that there are no non-diregular $(2,2,+2)$-digraphs, $(2,2,+2)$-digraphs are now classified up to isomorphism.  These conclusions have been verified computationally by Erskine \cite{Ers}.  It is interesting to note that neither of the $(2,2,+2)$-digraphs are vertex-transitive, for in each case there are exactly three vertices contained in two 3-cycles.  However, there does exist a Cayley $(2,2,+5)$-digraph (on the alternating group $A_4$), so it would be interesting to determine the smallest vertex-transitive $(2,2,+\epsilon )$-digraphs.

\section{Main result}

We can now complete our analysis by showing that there are no diregular $(2,k,+2)$-digraphs for $k \geq 3$.  Let $G$ be such a digraph.  By Lemma \ref{vertices with one common neighbour}, $G$ contains vertices $u$ and $v$ with a unique common out-neighbour.  In accordance with our vertex-labelling convention, we have the situation in Figure \ref{fig:kgeq3setup}.  A triangle based at a vertex $x$ represents the set $T(x)$.

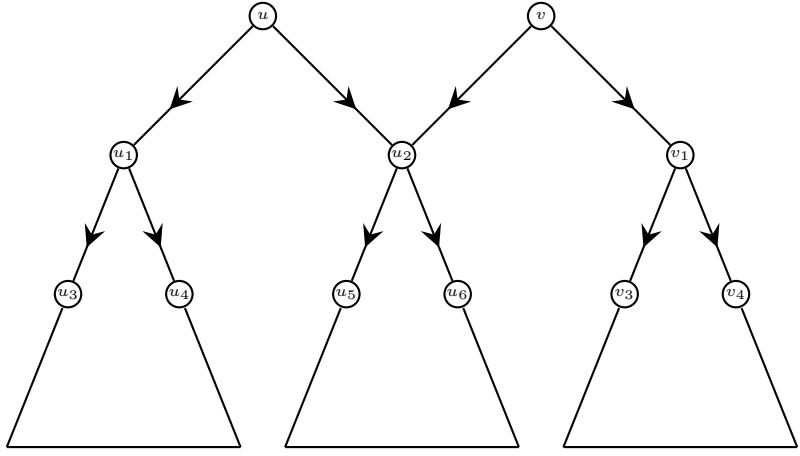
\begin{figure}\centering
\begin{tikzpicture}[middlearrow=stealth,x=0.2mm,y=-0.2mm,inner sep=0.1mm,scale=1.85,
	thick,vertex/.style={circle,draw,minimum size=10,font=\tiny,fill=white},edge label/.style={fill=white}]
	\tiny
	\node at (150,0) [vertex] (v0) {$u$};
	\node at (100,50) [vertex] (v1) {$u_1$};
	\node at (200,50) [vertex] (v2) {$u_2$};
	\node at (300,50) [vertex] (v3) {$v_1$};
	\node at (120,100) [vertex] (v4) {$u_4$};
           \node at (180,100) [vertex] (v5) {$u_5$};
	\node at (220,100) [vertex] (v6) {$u_6$};
	\node at (250,0) [vertex] (v7) {$v$};
          \node at (80,100) [vertex] (v9) {$u_3$};
           \node at (280,100) [vertex] (v10) {$v_3$};
          \node at (320,100) [vertex] (v11) {$v_4$};

	\path
		(v0) edge [middlearrow] (v1)
                     (v0) edge [middlearrow] (v2)
                     (v1) edge [middlearrow] (v9)
                     (v1) edge [middlearrow] (v4)
                     (v2) edge [middlearrow] (v5)
                     (v2) edge [middlearrow] (v6)
                     (v7) edge [middlearrow] (v2)
                     (v7) edge [middlearrow] (v3)
                     (v3) edge [middlearrow] (v10)
                     (v3) edge [middlearrow] (v11)
                    ;
                   \draw (78,105) -- (58,155);
                   \draw (122,105) -- (142,155);
                   \draw (178,105) -- (158,155);
                   \draw (222,105) -- (242,155);
                   \draw (278,105) -- (258,155);
                   \draw (322,105) -- (342,155);
                   \draw (58,155)--(142,155);
                   \draw (158,155)--(242,155);
                   \draw (258,155)--(342,155);

\end{tikzpicture}
\caption{Configuration for $k \geq 3$}
\label{fig:kgeq3setup}
\end{figure}

We now proceed to determine the possible outlier sets of $u$ and $v$.

\begin{lemma}\label{position of v}
$v \in N^{k-1}(u_1) \cup O(u)$ and $u \in N^{k-1}(v_1) \cup O(v)$.  If $v \in O(u)$, then $u_2 \in O(u_1)$ and if $u \in O(v)$, then $u_2 \in O(v_1)$.
\end{lemma}
\begin{proof}
$v$ cannot lie in $T(u)$, or the vertex $u_2$ would be repeated in $T_k(u)$.  Also, $v \not \in T(u_2)$, or there would be a $\leq k$-cycle through $v$.  Therefore, if $v \not \in O(u)$, then $v \in N^{k-1}(u_1)$.  Likewise for the other result.  If $v \in O(u)$, then neither in-neighbour of $u_2$ lies in $T(u_1)$, so that $u_2 \in O(u_1)$.
\end{proof}

\begin{lemma}\label{contraction lemma}
Let $w \in T(v_1)$, with $d(v_1,w) = l$.  Suppose that $w \in T(u_1)$, with $d(u_1,w) = m$.  Then either $m \leq l$ or $w \in N^{k-1}(u_1)$.  A similar result holds for $w \in T(u_1)$.
\end{lemma}
\begin{proof}
Let $w$ be as described and suppose that $m > l$.  Consider the set $N^{k-m}(w)$.  By construction, $N^{k-m}(w) \subseteq N^k(u_1)$, so by $k$-geodecity $N^{k-m}(w) \cap T(u_1) = \varnothing $.  At the same time, we have $l+k-m \leq k-1$, so $N^{k-m}(w) \subseteq T(v_1)$.  This implies that $N^{k-m}(w) \cap T(v_2) = N^{k-m}(w) \cap T(u_2) = \varnothing $.  As $V(G) = \{ u\} \cup T(u_1) \cup T(u_2) \cup O(u)$, it follows that $N^{k-m}(w) \subseteq \{ u\} \cup O(u)$.  Therefore $|N^{k-m}(w)| = 2^{k-m} \leq 3$.  By assumption $0 \leq m \leq k-1$, so it follows that $m = k-1$.
\end{proof}

\begin{corollary}\label{contraction corollary}
If $w \in T(v_1)$, then either $w \in \{ u\} \cup O(u)$ or $w \in T(u_1)$ with $d(u_1,w) = k-1$ or $d(u_1,w) \leq d(v_1,w)$. 
\end{corollary}
\begin{proof}
By $k$-geodecity and Lemma \ref{contraction lemma}.
\end{proof}

\begin{corollary}\label{position of v1}
$v_1 \in N^{k-1}(u_1) \cup O(u)$ and $u_1 \in N^{k-1}(v_1) \cup O(v)$.
\end{corollary}
\begin{proof}
We prove the first inclusion.  By Corollary \ref{contraction corollary}, $v_1 \in \{ u\} \cup O(u) \cup \{ u_1\} \cup N^{k-1}(u_1)$.  By $k$-geodecity, $v_1 \not = u$ and by construction, $v_1 \not = u_1$.
\end{proof}

We now have enough information to identify one member of $O(u)$ and $O(v)$.

\begin{lemma}\label{v1 in Ou}
$v_1 \in O(u)$ and $u_1 \in O(v)$.
\end{lemma}
\begin{proof}
We prove that $v_1 \in O(u)$.  Suppose that neither $v_1$ nor $v$ lies in $O(u)$.  Then by Lemma \ref{position of v} and Corollary \ref{position of v1} we have $v, v_1 \in N^{k-1}(u_1)$.  As $v_1$ is an out-neighbour of $v$, it follows that $v_1$ appears twice in $T_k(u_1)$, violating $k$-geodecity.  Therefore $O(u) \cap \{ v,v_1\} \not = \varnothing $.

Now assume that $v_1, v_3 \in T_k(u)$.  Again by Corollary \ref{position of v1}, $v_1 \in N^{k-1}(u_1)$.  By $k$-geodecity we also have $v_3 \in T(u_1)$.  However, $v_3 \in N^+(v_1)$, so $v_3$ appears twice in $T_k(u_1)$, which is impossible.  Hence $O(u) \cap \{ v_1,v_3\} \not = \varnothing $.  Similarly, $O(u) \cap \{ v_1,v_4\} \not = \varnothing $.  In the terminology of the previous section, $G$ contains no bad pairs.  Therefore, if $v_1 \not \in O(u)$, then $\{ v,v_3,v_4\} \subseteq O(u)$.  Since these vertices are distinct, this is a contradiction and the result follows.

\end{proof}

Lemma \ref{v1 in Ou} allows us to conclude that for vertices sufficiently close to $v_1$ one of the potential situations mentioned in Corollary \ref{contraction corollary} cannot occur.

\begin{lemma}\label{can't be at distance k-1}
$T_{k-3}(v_1) \cap N^{k-1}(u_1) = T_{k-3}(u_1) \cap N^{k-1}(v_1) = \varnothing $.
\end{lemma}
\begin{proof}
Let $w \in T_{k-3}(v_1) \cap N^{k-1}(u_1)$.  Consider the position of the vertices of $N^+(w)$ in $T_k(u) \cup O(u)$.  As $v_1 \not \in N^+(w)$, it follows from Lemma \ref{v1 in Ou} that at most one of the vertices of $N^+(w)$ can be an outlier of $u$, so let us write $w_1 \in N^+(w) - O(u)$.  By $k$-geodecity, $w_1 \not \in T(u_1) \cup \{ u\} $.  Hence $w_1 \in T(u_2) = T(v_2)$.  However, $w_1$ also lies in $T(v_1)$, so this violates $k$-geodecity.  
\end{proof}

\begin{corollary}\label{scrunching}
There is at most one vertex in $T_{k-3}(v_1) - \{ v_1\} $ that does not lie in $T(u_1)$; for all other vertices $w \in T_{k-3}(v_1) - \{ v_1\} $, $d(u_1,w) = d(v_1,w)$.  A similar result for $T_{k-3}(u_1) - \{ u_1\} $ also holds.
\end{corollary}

\begin{lemma}\label{scrunching k = 3}
For $k = 3$, $N^+(u_1) \cap N^2(v_1) = N^+(v_1) \cap N^2(u_1) = \varnothing $.
\end{lemma}
\begin{proof}
Suppose that $v_3 = u_7$.  By the reasoning of Lemma \ref{can't be at distance k-1} we can set $u = v_7$ and $O(u) = \{ v_1,v_8\} $.  $v \not \in O(u)$ and by 3-geodecity $v \not \in N^+(u_3)$, so we can assume that $v = u_9$.  $u_3 \rightarrow v_3$ implies that $u_3 \not \in T(v_1)$, so $O(v) = \{ u_1,u_3\} $.  We must have $\{ u_4,u_8,u_{10}\} = \{ v_4,v_9,v_{10}\} $.  As $u_4 \rightarrow v$, it follows that $v_4 = u_8$ and hence $\{ u_4,u_{10}\} = \{ v_9,v_{10}\} $, which is impossible.  
\end{proof}

As $u_1$ is an outlier of $v$, neither $v_3$ nor $v_4$ can be equal to $u_1$.  It follows from Corollary \ref{scrunching} and Lemma \ref{scrunching k = 3} that either $N^+(u_1) = N^+(v_1)$ or $u_1$ and $v_1$ have a single common out-neighbour, with one vertex of $N^+(v_1)$ being an outlier of $u$.

\begin{lemma}\label{N2u not equal to N2v}
$N^2(u) \not = N^2(v)$
\end{lemma}
\begin{proof}
Let $N^2(u) = N^2(v)$, with $N^+(u_1) = N^+(v_1) = \{ u_3,u_4\} $.  Suppose that $v \not \in O(u)$.  By Lemma \ref{position of v}, $v \in N^{k-2}(u_3) \cup N^{k-2}(u_4)$.  But then there is a $k$-cycle through $v$.  It follows that $O(u) = \{ v,v_1\}, O(v) = \{ u,u_1\} $.  By Lemma \ref{position of v}, $u_2 \in O(u_1) \cap O(v_1)$.  Therefore by Lemma \ref{identical neighbourhoods lemma} $O(u_1) = \{ u_2,v_1\} , O(v_1) = \{ u_2,u_1\} $.

Consider the in-neighbour $u'$ of $u_1$ that is distinct from $u$.  We have either $|N^+(u') \cap N^+(u)| = 1$ or $|N^+(u') \cap N^+(u)| = 2$.  In the first case, it follows from Lemma \ref{v1 in Ou} that $u_2 \in O(u')$.  Every vertex of $G$ is an outlier of exactly two vertices, so $u' = u_1$ or $v_1$.  In either case, we have a contradiction.  Therefore $N^+(u') = N^+(u)$.  It now follows from Lemma \ref{identical neighbourhoods lemma} that $u' \in O(u) = \{ v,v_1\} $, which is impossible.

\end{proof}
Noticing that $u_1$ and $v_1$ also have a unique common out-neighbour, we have the following corollary.

\begin{corollary}\label{outlier sets of u and u1}
Without loss of generality, $u_3 = v_3, u_9 = v_9, O(u) = \{ v_1,v_4\} , O(v) = \{ u_1,u_4\} , O(u_1) = \{ v_4,v_{10}\} $ and $O(v_1) = \{ u_4,u_{10}\} $.
\end{corollary}

We are now in a position to complete the proof by deriving a contradiction.

\begin{theorem}
There are no diregular $(2,k,+2)$-digraphs for $k \geq 3$.
\end{theorem}
\begin{proof}
$u,v \not \in \{ u_1,u_4,v_1,v_4\} $, so by Lemma \ref{position of v} $d(u,v) = d(v,u) = k$.  In fact, $u_3 = v_3$ implies that $v \in N^{k-2}(u_4)$ and $u \in N^{k-2}(v_4)$.  Let $k \geq 4$.  Then $u,v \not \in \{ u_{10},v_{10}\} $, so $u,v \in T_k(u_1) \cap T_k(v_1)$.  If $u \in T(u_3) = T(v_3)$, then $u$ would appear twice in $T_k(v_1)$, so $u \in N^{k-1}(u_4)$.  However, as $u$ and $v$ have a common out-neighbour, this violates $k$-geodecity.   

Finally, suppose that $k = 3$.  The above analysis will hold unless $u = v_{10}$ and $v = u_{10}$.  Let $N^-(u_1) = \{ u,u'\} , N^-(v_1) = \{ v,v'\} $.  It is evident that $v' \not \in \{ v_1,v_4\} $, so that $v' \in T_3(u)$.  As $v \in N^+(u_4)$, we must have $v' \in N^2(u_2)$.  Similarly $u' \in N^2(u_2)$.  Since $u_1$ and $v_1$ have a common out-neighbour, we can assume that $u' \in N^+(u_5)$ and $v' \in N^+(u_6)$.  $v_4$ can be the outlier of only two vertices, namely $u$ and $u_1$, so $v_4 \in N^3(u_2)$ and likewise $u_4 \in N^3(u_2)$.  By $3$-geodecity $v_4 \in N^2(u_5)$ and $u_4 \in N^2(u_6)$.  It follows that $u, v \not \in N^3(u_2)$, so $u \not \in T_3(u_1) \cup T_3(u_2)$.  Hence $O(u) = N^-(u) = \{ v_1,v_4\} $, which again is impossible.

\end{proof}

It is interesting to note that a similar argument can be used to provide an alternative proof of the result of \cite{Sil}.

\end{document}